\documentclass[a4paper,10pt]{scrartcl}
\usepackage[utf8]{inputenc}
\usepackage[english]{babel}
\usepackage[fixlanguage]{babelbib}
\usepackage{cite}
\usepackage{amssymb, amsmath, amstext, amsopn, amsthm, amscd, amsxtra, amsfonts}
\usepackage{yfonts, mathrsfs}
\usepackage{paralist}
\usepackage{bbm}

\usepackage{todonotes}

\usepackage{enumerate}

\usepackage[all]{xy}

\usepackage[T1]{fontenc}
\usepackage{ifthen}
\usepackage{colonequals}

\usepackage[bookmarks,bookmarksnumbered,plainpages=false]{hyperref}
\newtheoremstyle{standard}
{16pt}  
{16pt}  
{}  
{}  
{\bfseries}
{}  
{ } 
{{\thmname{#1~}}{\thmnumber{#2.}}\thmnote{~(#3)}} 

\newtheoremstyle{Lthm}
{\topsep}
{\topsep}
{\slshape}
{}
{\bfseries}
{}
{ }
{\thmname{#1}\thmnumber{ #2}\thmnote{  (#3).}}
\newtheoremstyle{kursiv}
{16pt}  
{16pt}  
{\itshape}  
{}  
{\bfseries}
{}  
{ } 
{{\thmname{#1~}}{\thmnumber{#2.}}\thmnote{~(#3)}}

\theoremstyle{kursiv}
\newtheorem{theorem}{Theorem}[section]
\newtheorem{lemma}[theorem]{Lemma}
\newtheorem{proposition}[theorem]{Proposition}
\newtheorem{corollary}[theorem]{Corollary}

\theoremstyle{Lthm}
\newtheorem{Ltheorem}{Theorem}

\theoremstyle{definition}
\newtheorem{definition}[theorem]{Definition}
\newtheorem{remark}[theorem]{Remark}
\newtheorem{example}[theorem]{Example}
\newtheorem{notation}[theorem]{Notation}  
\newtheorem{question}{Question}

\renewcommand{\phi}{\varphi}

\setdefaultenum{\normalfont (a)}{i.}{A.}{1.}

\renewcommand{\epsilon}{\varepsilon}

\DeclareSymbolFont{bbold}{U}{bbold}{m}{n}
\DeclareSymbolFontAlphabet{\mathbbold}{bbold}

\title{On the Topology of J-Groups} 
\author{Rafael Dahmen\footnote{Karlsruhe Institut für Technologie (KIT), Karlsruhe, Germany. 
\href{mailto:rafael.dahmen@kit.edu}{rafael.dahmen@kit.edu}, ORCID: \href{https://orcid.org/0000-0002-7927-8931}{0000-0002-7927-8931}}
\date{September 9, 2022}
}

\begin{document}
  
  \maketitle
  
  \begin{abstract}
    We introduce the concept of a topological J-group and determine for many important examples of topological groups if they are topological J-groups or not. 
    Besides other results, we show that the underlying topological space of a pathwise connected topological J-group is weakly contractible which is a strong and unexpected obstruction that depends only on the homotopy type of the underlying space.
  \end{abstract}
  
  \medskip
  
  
  \medskip
  
  (secondary)
  
  \medskip
  
  \textbf{Statements and Declarations:}
  This research received no external funding. The author declares no conflict of interest.

 \section{Introduction and statement of the results}
  A \emph{J-group} $(G,w,f)$ is a group $G$, together
  with a group element $w\in G$ (called the \emph{witness})
  and a self-map $ f \colon G \rightarrow G$ such that for all $x\in G$ one has
  \[
     f(x\cdot w) = f(x)\cdot x\,.
  \]
  This property\footnote{In fact, they use the definition $f(x\cdot w)=x\cdot f(x)$ but for (abstract) groups this is equivalent to the definition given here (by \cite[Remark 2.5]{JGroups}).} appeared first in \cite{BoykettWendt} under the name \emph{Property X} where it arose naturally in the study of the Jacobson radical of certain nearrings that are related to state machines.
  
  The name \emph{J-group} was introduced in \cite{JGroups} where the authors investigate this property more closely.
  They find necessary and sufficient conditions
  for a given group $G$ to have a J-group structure. This turns out to be a nontrivial task, involving surprisingly demanding combinatorial and number theoretical aspects.
  
  In this article we will transfer the purely algebraic idea of J-groups to the topological category by requiring the map $f$ to be continuous (see Definition \ref{definition_topological_j_group}).
  There may not be an obvious reason to do so, but it turns out that (although there are some easy examples of topological J-groups, like $( \mathbb{R},+)$ (see Lemma  \ref{lemma_j_field})), the question whether a given topological group
  admits the topological analog of a J-group structure is an interesting one, combining techniques from topological algebra, number theory, algebraic topology and functional analysis.
  
  Possible applications of this concept to topological nearrings and continuous state machines are not part of this article, but seem to be an interesting and promising question for a future work.
  
  Our main result is the following surprising theorem on pathwise connected topological J-groups:
  \begin{Ltheorem}[Homotopy of J-Groups]\ \\            \label{thm_HOMOTOPY}
    A pathwise connected topological J-group is weakly contractible, i.e. all homotopy groups are trivial.
  \end{Ltheorem}
  Note that this gives a necessary condition for a topological group
  to carry a topological J-group structure that depends only on the homotopy type of the underlying topological space, ruling out many classes of topological groups, like $ \mathrm{GL}_{n}\!\left(   \mathbb{C}    \right)$, $ \mathrm{SL}_{n}\!\left(   \mathbb{R}    \right)$, $ \mathrm{SU}_{n}\!\left(   \mathbb{C}    \right),  \mathrm{U}_{n}\!\left(   \mathbb{C}    \right), \ldots$.
  
  The converse of Theorem \ref{thm_HOMOTOPY} is false, i.e. there are weakly contractible topological groups which are not topological J-groups (see Example \ref{example_boolean_group}).
  For connected Lie Groups, as was pointed out to the author by Karl-Hermann Neeb, the converse turns out to be correct, see Corollary \ref{corollary_contractible_lie_groups}.

  In the class of compact groups we can show the following necessary criterion for being a topological J-group which does not depend on the condition of pathwise connectedness:
  \begin{Ltheorem}[Compact Abelian J-groups] \ \\               \label{thm_COMPACT}
    Every compact abelian metrizable topological J-group is profinite, i.e. totally disconnected.
  \end{Ltheorem}
  Theorem \ref{thm_COMPACT} can also be used to show that many topological groups which one encounters in topological algebra cannot be topological J-groups.
  For example, the $p$-adic solenoid (see e.g. \cite[Example 1.28  (ii)]{CompBook} for a definition) is not a topological J-group. This does not follow from Theorem \ref{thm_HOMOTOPY} since the solenoid is not pathwise connected.
  
  Although we need metrizability and commutativity for the proof of Theorem \ref{thm_COMPACT} to work, the author conjectures that these assumptions are really not necessary:
  \begin{question}        \label{question_compact_J_group}
    Is it true that every compact topological J-group is profinite?
  \end{question}
  
  In \cite{JGroups}, the authors reduce the problem of deciding whether a given abstract group carries a J-group structure to torsion groups by showing the following (see \cite[Proposition 2.2]{JGroups}):
  \begin{proposition}\label{proposition_infinite_order_J_group}
    If a (discrete) group $G$ contains an element $w\in G$ of infinite order, then there is a map $ f \colon G \rightarrow G$ such that $(G,w,f)$ is a J-group with witness $w$.
  \end{proposition}
  A natural question would be if a similar statement holds for topological J-groups as well. 
  However, both, Theorem \ref{thm_HOMOTOPY} and Theorem \ref{thm_COMPACT}, 
  show that the circle group is not a topological J-group although it has many elements of infinite order. The closest we have as a topological analog is Theorem \ref{thm_neeb} which has strong additional topological assumptions on the underlying space.
  
  Furthermore, in \cite{JGroups}, the authors conjecture that a finite nilpotent group is a J-group if and only if it is of odd order. The assumption of having odd order is clearly necessary as every group with even order cannot be a J-group by \cite[Corollary 2.5]{JGroups}.
  
  We state and prove a similar statement for connected Lie groups\footnote{A  ``Lie group'' in this article is always finite-dimensional and defined over the field of real numbers.}:
  \begin{Ltheorem}[Nilpotent J-Groups]\ \\ \label{thm_NILPOTENT}
    Let $G$ be a nilpotent connected Lie group. Then $G$ is a topological J-group if and only if $G$ is simply connected.
  \end{Ltheorem}
  The class of connected Lie groups behaves in many aspects similar to the class of finite groups. The additional assumption of being simply connected is necessary as the fundamental group of a pathwise connected topological J-group has to vanish by Theorem \ref{thm_HOMOTOPY}.
  
  In \cite{JGroups}, it is shown that every finite J-group is solvable. In our topological context, the author conjectured at first that a similar statement should hold for connected Lie groups as well, but it turns out that the universal cover of the special linear group $ \mathrm{SL}_{2}\!\left(   \mathbb{R}    \right)$ is a connected Lie group which is a topological J-group without being solvable. In a certain way, this is the only counter-example as we can show that a topological J-group which is a connected Lie group always decomposes as a semidirect product of a solvable Lie group and a finite number of copies of the universal cover of $ \mathrm{SL}_{2}\!\left(   \mathbb{R}    \right)$ (see Corollary \ref{corollary_contractible_lie_groups}).
  
  Theorem \ref{thm_HOMOTOPY} suggests that there could be interesting examples for topological J-groups among the contractible groups. Famous examples of contractible abelian topological groups are real topological vector spaces. Here we can prove the following positive result:
  \begin{Ltheorem}[Topological Vector J-Groups]        \label{thm_TVS}\ \\
    Let $E$ be a real Hausdorff topological vector space satisfying at least one of the following:
    \begin{itemize}
     \item $E$ is locally convex.
     \item $E$ is paracompact.
     \item The topological dual of $E$ is not trivial.
    \end{itemize}
    Then the additive group of $E$ is a topological J-group.
  \end{Ltheorem}
  As real normed vector spaces are locally convex (and paracompact), this includes all Hilbert spaces and all Banach spaces, in particular.
  \begin{question}         \label{question_TVS}
    Are the hypotheses in Theorem \ref{thm_TVS} really needed?
    The argument for the locally convex case is totally different from the argument for the paracompact case; so it may be true that there exists an argument that does not need any of those assumptions.
    
    However, it seems that there are not so many interesting examples of Hausdorff topological vector spaces that are not paracompact and have a trivial dual space. 
    In particular, the topological vector space $L^p([0,1], \mathbb{R})$ with $0<p<1$ which is possibly the most prominent example of a topological vector space with a trivial dual (see e.g. \cite[Theorem 1]{Day1940}) is metrizable and hence paracompact.
  \end{question}

  Theorem \ref{thm_COMPACT} suggests that there could be interesting examples of topological J-groups in the world of profinite groups.  Naturally, one cannot expect the underlying space to tell us much about the group structure as many profinite groups are topologically just Cantor sets. 
  However, one can say the following:
  
  \begin{Ltheorem}[Abelian Profinite J-Groups]      \ \\ \label{thm_PROFINITE}
    Let $G$ be an abelian profinite group.
    \begin{itemize}
     \item [(a)]    If $G$ is torsion-free then $G$ is a topological J-group. 
     \item [(b)]   If $G$ is a torsion group then $G$ is a topological J-group if and only if the order of every element is an odd number.
    \end{itemize}
  \end{Ltheorem}
  This includes the group of $p$-adic integers for every prime number $p$, as well as the profinite completion of $ \mathbb{Z}$ as new examples of topological J-groups.
  
  Although Theorem \ref{thm_PROFINITE} might look like a complete answer to the question whether an abelian profinite group is a topological J-group, unfortunately, there are abelian profinite groups where this theorem cannot be applied:
  
  \begin{question}        \label{question_INFINITE_PRODUCT}
    Is it true that the infinite product
    \[
       G = \prod_{j=1}^\infty  \mathbb{Z}/(2^k \mathbb{Z})
    \]
    is a topological J-group?\footnote{For odd primes, the corresponding product is a topological J-group since it can be written as a product of topological J-groups.}
  \end{question}
  The profinite group in Question \ref{question_INFINITE_PRODUCT} 
  contains a dense torsion subgroup but is itself not a torsion group, hence both parts of Theorem
  \ref{thm_PROFINITE} cannot be applied.
  Interestingly, it contains a closed copy of the group of $2$-adic integers, which is a topological J-group.\\
  This shows that deciding whether an abelian profinite group is a topological J-group is difficult. For nonabelian groups, the situation is even worse since this class includes the nonabelian finite case which has still many open problems (see \cite{JGroups}). It would be interesting to know whether some of the methods of finite J-groups in \cite{JGroups} can be modified and carried over to profinite J-groups.
  This is the essence of a current project of the author together with Johannes Flake.

  This article is structured as follows: 
  After introducing basic definitions and proving some elementary facts about topological J-groups in Section \ref{section_basic_definitions},
  we deal with some necessary conditions for a topological group to be a topological J-group in Sections \ref{section_homotopy} and \ref{section_compact}.
  In Section \ref{section_homotopy}, we introduce facts about (free) loop groups and homotopy groups and use them to prove Theorem \ref{thm_HOMOTOPY} and other necessary conditions for a topological group to be a topological J-group that only depend on the homotopy type of the underlying space. In Section \ref
  {section_homotopy} you can also find the proof of Theorem \ref{thm_NILPOTENT}. 
  In Section \ref{section_compact}, we introduce the notion of an exponent sequence of a compact group and use it to prove Theorem \ref{thm_COMPACT}.
  
  Afterwards, in Section \ref{section_j_rings} we use  the concept of a topological J-ring to construct examples of topological J-groups. In particular, we prove Theorems \ref{thm_TVS} and \ref{thm_PROFINITE} stated in the introduction.
  
  There are open problems spread over the whole article, labeled with the word \emph{Question}.

 \section{Basic Definitions and First Results}
  \label{section_basic_definitions}
  \begin{notation}
    All topological spaces are assumed to be Hausdorff, all rings are unital, but not necessarily commutative. The set of natural numbers $ \mathbb{N}=  \left\{ 1,2,3,\ldots \right\} $ starts with $1$; we set $ \mathbb{N}_0:= \mathbb{N}\cup  \left\{ 0 \right\} $. The symbols $ \mathbb{Z}, \mathbb{R}, \mathbb{C}$ denote the sets of integers, 
    reals and complex numbers, endowed with their usual algebraic and topological structure.
    The circle group is denoted by $ \mathbb{T}:= \mathbb{R}/ \mathbb{Z}$ and written additively but is of course topologically isomorphic to $S^1:=(  \left\{ z\in \mathbb{C} \ :\  \left| z \right|=1\right\} ,\cdot)$. If $G$ is a group, the neutral element of $G$ is denoted by $1_G$. 
    A \emph{Lie group} is always finite-dimensional and defined over the field of real numbers, a \emph{Lie algebra} may be defined over any field and can have any dimension.
    
    For a prime number $p$ the symbol $ \mathbb{Z}_p$ denotes the topological ring of $p$-adic  integers (not the finite ring $ \mathbb{Z}/p \mathbb{Z}$) and $ \mathbb{Q}_p$ denotes the topological field of $p$-adic rationals.
    
    For an integer $n\in \mathbb{Z}$ the \emph{second binomial coefficient} is defined as
    \[
       \binom{n}{2} = 2^{-1}\cdot n \cdot (n-1)\in  \mathbb{Z} \text{ for all }n\in  \mathbb{Z}\,.
    \]
    This definition can also be applied to rings in which $2=1+1$ is an invertible element (see Lemma \ref{lemma_binomial_j_ring}).
  \end{notation}
  
  \begin{definition}[Topological J-Group]
    \label{definition_topological_j_group}
    A \emph{topological J-group} is  a triple $(G,w,f)$,
    where $G$ is a Hausdorff topological group,
    $w\in G$ is an element of $G$ (called the \emph{witness}) and $ f \colon G \rightarrow G$ is a
    continuous self-map such that for every $x\in G$ we have
    \[
       f(x\cdot w)=f(x)\cdot x\,.
    \]
    By slight abuse of terminology, we will call  $G$ a (topological) J-group if there exist $w,f$ such that $(G,w,f)$ is a (topological) J-group.
  \end{definition}
  Clearly, a group which is a J-group in the sense of \cite{JGroups} is also a topological J-group in the sense of Definition \ref{definition_topological_j_group} when endowed with the discrete topology.

  \begin{lemma}       \label{lemma_x_plus_nw}
    Let $(G,w,f)$ be a topological J-group and let $x\in G$ be an element which commutes with the witness $w$.
    
    Then we have
    \[
       (\forall n\in \mathbb{Z})\quad 
       f(x\cdot w^n) = f(x)\cdot x^n\cdot w^{\binom{n}{2}}\,.
    \]
    In particular, we have that $f(w)=f(1_G)$ and $f(w^2)=f(1_G)\cdot w$.
  \end{lemma}
  This is a special case of the more general statement \cite[Lemma 2.1]{JGroups}, but we will give a direct proof for the reader's convenience:
  \begin{proof}[Proof of Lemma \ref{lemma_x_plus_nw}]
    First of all, let us note that the formulas for $f(w)$ and $f(w^2)$ follow from the general formula by setting $x=1_G$ and $n=1,2$.
    
    Let $x\in G$ be an element which commutes with the witness $w$. 
    For each $n\in \mathbb{Z}$ consider the statement
    \[
       P(n):\iff\quad              
       \framebox{$f(x\cdot w^n) = f(x)\cdot x^n\cdot w^{\binom{n}{2}}$}\,.
    \]
    Since $P(0)$ is clearly true, all we have to show is that
    \[
       (\forall n\in \mathbb{Z})\quad               \bigl( P(n+1)\iff P(n)     \bigr)
    \]
    and by a variant of mathematical induction that works on the integers, it will follow that $P(n)$ is true for all $n\in \mathbb{Z}$. 
    
    \begin{align*}
      P(n+1)
      &\iff f(x\cdot w^{n+1}) = f(x)\cdot x^{n+1}\cdot w^{\binom{n+1}{2}}\\
      &\iff f((x\cdot w^n)\cdot w)      = f(x)\cdot x^n\cdot x\cdot w^{\binom{n}{2}+n}\\
      &\iff f(x\cdot w^n)\cdot (x\cdot w^n)   
      = f(x)\cdot x^n\cdot x\cdot w^{\binom{n}{2}}\cdot w^n \\
      &\iff f(x\cdot w^n)\cdot (x\cdot w^n)     
      =  f(x)\cdot x^n\cdot w^{\binom{n}{2}}\cdot 
      (x\cdot w^n)\\
      &\iff f(x\cdot w^n)    
      =  f(x)\cdot x^n\cdot w^{\binom{n}{2}}\\
      &\iff P(n)\,.
      \qedhere
    \end{align*}
  \end{proof}
  
  \begin{lemma}           \label{lemma_order_of_w}
    Let $(G,w,f)$ be a topological J-group. If $w$ is of finite order $n$, then the order of every element which commutes with $w$ is finite and is a divisor of $n$.
    Furthermore, $n$ is an odd number.
  \end{lemma}
  \begin{proof}
    Assume that $w$ is of finite order $n\in \mathbb{N}$.
    Then by Lemma \ref{lemma_x_plus_nw} we have for all $x\in G$ which commute with $w$:
    \[
       f(x)\cdot x^n\cdot w^{\binom{n}{2}} =  f(x\cdot w^n)\,.
    \]
    Since $w^n=1_G$, this can be reduced to
    \begin{equation}
      x^n\cdot w^{\binom{n}{2}} = 1_G\text{ for all }x\in G \text{ which commute with }w.   \label{eqn_nx_plus_binom}
    \end{equation}
    
    If we first look at the case $x=1_G$, we obtain
    \[
       w^{\binom{n}{2}}  = 1_G
    \]
    which means that $\binom{n}{2}$ has to be a divisor of the order of $w$, i.e. there is a $k\in \mathbb{Z}$ such that
    \[
       \binom{n}{2} = kn\,.
    \]
    Multiplying both sides by $2$ and dividing by $n$ yields $n-1=2k$ and hence $n$ is an odd number.
    
    Let us return to a general $x\in G$ commuting with $w$.
    Plugging in $w^{\binom{n}{2}}=1_G$ into equation (\ref{eqn_nx_plus_binom}), we obtain
    $x^n = 1_G$ which implies that the order of $x$ is finite and a divisor of $n$. This completes the proof.
  \end{proof}
  Lemma \ref{lemma_order_of_w} implies that if the witness $w$ is in the center of $G$, it is of maximal order. 
  One could conjecture that the witness of a topological J-group always lies in the center of the group. 
  This, however, is already false for discrete J-groups as there are examples of finite J-groups where the order of the witness is not the maximal element order (see \cite[Example 5.5]{JGroups}) and hence the witness is not central. Also, for connected Lie groups this is false as in the nonabelian topological J-group $\widetilde{ \mathrm{SL}_{2}\!\left(   \mathbb{R}    \right)}$ every nontrivial element can be a witness by Corollary \ref{corollary_contractible_lie_groups}. 
  There are even connected topological J-groups with a trivial center (see Example \ref{example_pronilpotent_group_with_trivial_center}).

  Considering the definition of a topological J-group one could ask whether it is possible that the self-map $f$ is a group homomorphism or that $w$ is the identity element. Also one could ask if $f$ can be a bijection, i.e. a permutation of the group. Trivially, the trivial group $  \left\{ 1_G \right\} $ is a topological J-group with 
  all these nice properties. Interestingly, this is also the only group with these properties as the following shows:
  \begin{proposition}\label{proposition_trivial_J_group}
    Let $(G,w,f)$ be a topological J-group. Then the following are equivalent:
    \begin{itemize}
     \item [(i)] $G$ is trivial.
     \item [(ii)] $f$ is a group homomorphism.
     \item [(iii)] $w=1_G$
     \item [(iv)] $f$ is injective.
    \end{itemize}
  \end{proposition}
  \begin{proof}
    It is clear that (i) implies (ii).
    
    We will show that (ii) implies (iii): 
    Since $f$ is a group homomorphism, we have $f(1_G)=1_G$ and $f(w^2)=(f(w))^2$. If we combine these two statements with $f(w)=f(1_G)$ and $f(w^2)=f(1_G)\cdot w$ from Lemma \ref{lemma_x_plus_nw}, we easily deduce that $w=1_G$.
    
    Let us show now that (iii) implies (i). To this end, let us assume that $w$ is the identity element of the group. Then $w$ commutes with all the elements of the group and the order of $w$ is $1$. By Lemma \ref{lemma_order_of_w}, the order of every element has to divide $1$. So, $G$ is trivial.
    
    This shows that (i), (ii), and (iii) are equivalent. 
    Again, by Lemma \ref{lemma_x_plus_nw}, we know that $f(w)=f(1_G)$, so (iv) implies (iii). And it is clear that (iv) follows from (i).
  \end{proof}
  
  It should be noted that there are nontrivial topological J-groups in which the map $f$ is surjective, e.g. the group $( \mathbb{C},1,f)$ with $f(z)=\binom{z}{2}=\frac12z(z-1)$ (see Lemma \ref{lemma_j_field}).
  
  We end this section by stating a few closure properties of the class of J-groups:
  \begin{remark} \label{remark_closure_properties}
    \begin{itemize}
     \item [(a)] The product of any family of topological J-groups is again a topological J-group. 
     \item [(b)] In general, neither quotients nor closed subgroups of topological J-groups are topological J-groups, as the example $G= \mathbb{Z}\times \mathbb{T}$ shows. The group $G$ is a topological J-group by Proposition \ref{proposition_Z_times_group}, but the direct factor $ \mathbb{T}$ (which is a quotient and a subgroup) is not a topological J-group by Theorems \ref{thm_HOMOTOPY} and \ref{thm_COMPACT}.
      Even subgroups of finite index do not need to be J-groups as there are examples of subgroups of finite J-groups which are not J-groups (see \cite[Remark 3.5]{JGroups}).
     \item [(c)] Direct limits of topological J-groups need not be topological J-groups. For a counter-example, fix an odd prime number $p$.
      For each $n\in \mathbb{N}$ consider the cyclic group 
      \[
         G_n:=  \left\{ \left[\frac{k}{p^n}\right]_ \mathbb{Z} \ :\ k\in \mathbb{Z}\right\} \subseteq  \mathbb{R}/ \mathbb{Z}
      \]
      of order $p^n$, regarded as a subgroup of the additive circle group $ \mathbb{T}= \mathbb{R}/ \mathbb{Z}$. 
      Each group $G_n$ is isomorphic to the additive group of the discrete ring $ \mathbb{Z}/p^n \mathbb{Z}$ which is a J-group by Lemma \ref{lemma_binomial_j_ring}.
      
      The system $ \left( G_n \right)_{n\in \mathbb{N}}$ is a directed family of discrete groups. The direct limit in the category of topological groups is the \emph{Prüfer group}
      \[
         G=\bigcup_{n\in \mathbb{N}}G_n =   \left\{ \left[\frac{k}{p^n}\right]_ \mathbb{Z} \ :\ k\in \mathbb{Z};  n\in \mathbb{N}\right\}       
      \]
      endowed with the discrete topology. This group cannot be a J-group since it is abelian but does not contain an element of maximal order, contradicting Lemma \ref{lemma_order_of_w}.
      
      The reason for this problem lies in the fact that the witnesses and the self-maps of $G_n$ are not compatible.

      This problem could be fixed by only considering direct limits of compatible systems of topological J-groups.
      Let $( \mathbb{I},\leq)$ be a directed set and consider a directed system $ \left( G_\alpha,w,f_\alpha \right)_{\alpha\in  \mathbb{I}}$ of ascending topological J-groups, where $w$ is the same for the whole directed system and the maps $f_\alpha$ are compatible in the sense that $f_\beta|_{G_\alpha}=f_\alpha$ whenever $\alpha\leq\beta$. Then the direct limit would be a topological J-group again, assuming the directed system has the \emph{algebraic colimit property} (see \cite{LongColimitsI}), i.e. the direct limit in the category of topological spaces agrees with the direct limit in the category of topological groups. Otherwise it is not clear that the map $f$ defined on the direct limit is continuous. 
    \end{itemize}
  \end{remark}
  
  \begin{question}\label{question_projective_limit}
    Is the projective limit of topological J-groups, taken in the category of topological groups again a topological J-group?
    
    The author conjectures that this is false in general\footnote{It is however possible that the projective limit of a family of topological groups which are not topological J-groups turns out to be a topological J-group, e.g. $ \mathbb{Z}_2=\lim_n  \mathbb{Z}/(2^n \mathbb{Z})$.} as there is no obvious way to construct a witness of the projective limit out of the witnesses of the original groups as the witnesses and the maps $f$ are not required to be compatible. However, with the correct compatibility requirements, it is easy to show that the projective limit of topological J-groups will be a topological J-group.
  \end{question}

  \begin{question}		\label{question_definition_J_group}
    We say that a topological group $G$ has property $(J_1),(J_2),(J_3),(J_4)$ respectively, if there is a constant $w$ and a continuous self-map $ f \colon G \rightarrow G$ such that
    \begin{itemize}
     \item [($J_1$)] $f(x\cdot w) = x\cdot f(x);$
     \item [($J_2$)] $f(x\cdot w) = f(x)\cdot x;$
     \item [($J_3$)] $f(w\cdot x) = x\cdot f(x);$
     \item [($J_4$)] $f(w\cdot x) = f(x)\cdot x.$
    \end{itemize}
    In Definition \ref{definition_topological_j_group}, we took property $(J_2)$ to define a topological J-group, but it is a natural question whether these different concepts all lead to the same class of topological groups. 
    It is easy to see $(J_1)$ is equivalent to $(J_4)$ and that $(J_2)$ is equivalent to $(J_3)$ (the witness $w$ and the self-map $f$ have to be modified). 
    In \cite[Remark 2.6]{JGroups} it is shown that for \emph{discrete} J-groups all four notions are equivalent. However, it seems that the methods cannot be carried over. In any way, the difference seems to be not so important for the results of the present article since in most of our concrete examples the witness $w$ lies in the center of the group and in this case $(J_1)$ is clearly equivalent to $(J_2)$.
  \end{question}

 \section{J-Groups and Homotopy}
  \label{section_homotopy}
  In this section we will find some necessary conditions on a topological group $G$ for being a topological J-group, that depend only on the homotopy type of the underlying space $G$ and not on the group structure. 
  
  Before we do that, let us recall some topological concepts and fix some notation:
  \begin{notation}
    Let $X$ be a topological space and $x\in X$. The connected component of $x$ in $X$ is denoted by $ \left[ x \right]_\mathrm{c}$; the path-component of $x$ in $X$ is denoted by $ \left[ x \right]_p$.
    
    For a topological space $X$ the space of connected component is denoted by $ \pi_\mathrm{c}\!\left(    X    \right)$ and endowed with the quotient topology. It is always totally disconnected.
    A continuous map $ f \colon X \rightarrow Y$ between topological spaces maps connected components of $X$ into connected components of $Y$, therefore it induces a continuous map 
    \[
       \pi_\mathrm{c}\!\left(    f    \right) \colon  \pi_\mathrm{c}\!\left(    X    \right) \rightarrow  \pi_\mathrm{c}\!\left(    Y    \right)\,.
    \]
    This defines a functor
    \[
       \pi_\mathrm{c}\!\left(    \cdot    \right) \colon  \mathbf{TOP} \rightarrow  \mathbf{TOP}    
    \]
    which is multiplicative and therefore maps topological groups to topological groups, i.e. $ \pi_\mathrm{c}\!\left(    G    \right)$ carries the structure of a topological group if $G$ is a topological group.\footnote{This is exactly the topological group one obtains by factoring out the connected component of the identity, see e.g. \cite{ComponentFactor}.}
    
    The set of path-components of $X$ is denoted by $ \pi_0\!\left(    X    \right)$ and endowed with the discrete topology\footnote{One could also take the quotient topology but this could be non-Hausdorff and we do not want to study non-Hausdorff groups.}.
    Analogously, we obtain a functor
    \[
       \pi_0\!\left(    \cdot    \right) \colon  \mathbf{TOP} \rightarrow  \mathbf{SET}    
    \]
    which takes topological groups to groups.
    
    Recall that for a topological space $X$ with basepoint $x_0$, the $n$-th \emph{homotopy group} can be defined recursively as
    \[
       \pi_0(X,x_0):= \pi_0\!\left(    X    \right);\quad\quad
       \pi_{n+1}(X,x_0):=\pi_n\bigl(\Omega (X,x_0)\bigr)\,,
    \]
    where $\Omega(X,x_0)$ is the space of closed loops based at $x_0$, endowed with the compact-open topology (see e.g. \cite[Section 4.3]{Hatcher}).
    
    We say a space $X$ is \emph{simply connected} if it is pathwise connected and $\pi_1(X,x_0)$ is trivial (the basepoint $x_0$ is irrelevant here).

    A topological group $G$ will always be considered a topological space with basepoint, where the basepoint is equal to the neutral element $1_G$.
  \end{notation}
  We will now introduce a method to create new topological J-groups from old ones:
  \begin{proposition} \label{proposition_pi_0}
    Let $(G,w,f)$ be a topological J-group.
    \begin{itemize}
     \item [(a)] The group of connected components 
      \[
         \pi_\mathrm{c}\!\left(    G    \right)=  \left\{  \left[ a \right]_\mathrm{c} \ :\ a\in G\right\} \,,
      \]
      endowed with the quotient topology,
      is a topological J-group.
     \item [(b)] The group of path-components 
      \[
         \pi_0\!\left(    G    \right)=  \left\{  \left[ a \right]_p \ :\ a\in G\right\} \,,
      \]
      endowed with the discrete topology,
      is a (topological) J-group.
    \end{itemize}
  \end{proposition}
  \begin{proof}
    (a)\\
    The continuous map $ f \colon G \rightarrow G$ induces a continuous map $  \pi_\mathrm{c}\!\left(    f    \right) \colon  \pi_\mathrm{c}\!\left(    G    \right) \rightarrow  \pi_\mathrm{c}\!\left(    G    \right)$ which satisfies
    \[
       \pi_\mathrm{c}\!\left(    f    \right)\left( \left[ x \right]_\mathrm{c}\cdot  \left[ w \right]_\mathrm{c}\right) =  \left[ f(x\cdot w) \right]_\mathrm{c}=  \left[ f(x)\cdot x \right]_\mathrm{c}=  \pi_\mathrm{c}\!\left(    f    \right)\left( \left[ x \right]_\mathrm{c}\right)\cdot \left[ x \right]_\mathrm{c}\,.
    \]
    
    The proof of part (b) is completely analogous.
  \end{proof}

  \begin{remark}
    Since $ \pi_0\!\left(    G    \right)$ is a J-group for each topological J-group $G$ by Proposition \ref{proposition_pi_0}, one could hope that the same holds for the fundamental group $\pi_1(G)$ of a topological J-group $G$.
    
    However, this is not true as the example $G= \mathbb{Z}\times  \mathrm{SO}_{3}\!\left(   \mathbb{R}    \right)$ shows. The group $G$ is a topological J-group by Proposition  \ref{proposition_Z_times_group}. However, its fundamental group
    \[
       \pi_1(G) =\pi_1\bigl( \mathbb{Z}\times  \mathrm{SO}_{3}\!\left(   \mathbb{R}    \right)\bigr)
       \cong\pi_1( \mathbb{Z})\times \pi_1( \mathrm{SO}_{3}\!\left(   \mathbb{R}    \right))
       \cong \pi_1( \mathrm{SO}_{3}\!\left(   \mathbb{R}    \right))
    \]
    has order $2$ (see e.g. \cite[Section 3.D; Exercise 2]{Hatcher}) and so it cannot be a J-group by Lemma \ref{lemma_order_of_w}.
  \end{remark}

  \begin{corollary}      \label{corollary_connected_component}
    Let $(G,w,f)$ be a topological J-group.
    \begin{itemize}
     \item [(a)] If the number of connected components of $G$ is finite, this number is odd.
     \item [(b)] If the number of path-components of $G$ is finite, this number is odd.
     \item [(c)] If $1_G$ and $w$ lie in the same connected component, then $G$ is connected.
     \item [(d)] If $1_G$ and $w$ lie in the same path-component, then $G$ is pathwise connected.
    \end{itemize}
  \end{corollary}
  \begin{proof}
    By Proposition \ref{proposition_pi_0} the groups $ \pi_\mathrm{c}\!\left(    G    \right)$ and $ \pi_0\!\left(    G    \right)$ are J-groups with witness $ \left[ w \right]_\mathrm{c}$, and $ \left[ w \right]_p$ respectively. 
    
    For (a) and (b), we apply \cite[Corollary 2.5]{JGroups} to the group $ \pi_\mathrm{c}\!\left(    G    \right)$, or $ \pi_0\!\left(    G    \right)$, respectively.
    
    For (c) and (d), we apply Proposition \ref{proposition_trivial_J_group} to the group $ \pi_\mathrm{c}\!\left(    G    \right)$, or $ \pi_0\!\left(    G    \right)$, respectively.
  \end{proof}
  
  \begin{example}
    Corollary \ref{corollary_connected_component} shows that the real multiplicative group $( \mathbb{R}\setminus  \left\{ 0 \right\} ,\cdot)\cong \mathbb{R}\times( \mathbb{Z}/2 \mathbb{Z})$ is not a topological J-group as it has two connected components. More generally, for each $n\in \mathbb{N}$, the real general linear group $ \mathrm{GL}_{n}\!\left(   \mathbb{R}    \right)$ has two connected components and hence is not a topological J-group. The same is true for the orthogonal groups $ \mathrm{O}_{n}\!\left(   \mathbb{R}    \right)$.
  \end{example}
  
  Now we continue constructing new topological J-groups:
  \begin{proposition}\label{proposition_mapping_J_groups}
    Let $D$ be a Hausdorff topological space and let $(G,w,f)$ be a topological J-group. 
    We denote by $C(D,G)$ the group of all continuous functions from $D$ to $G$, endowed with the compact-open topology and pointwise group operations.
    
    Let $C(D,w)$ be the constant function, mapping everything to $w$ and let $C(D,f)$ be the map
    \[
       C(D,f) \colon C(D,G) \rightarrow C(D,G) ,\ \gamma \mapsto f\circ\gamma\,.
    \]
    
    \begin{itemize}
     \item [(a)]  Then $\bigl(C(D,G),C(D,w),C(D,f))$ is a  topological J-group as well.
     \item [(b)] If $G$ is connected then $C(D,G)$ is connected as well.
     \item [(c)] If $G$ is pathwise connected then $C(D,G)$ is pathwise connected as well.
    \end{itemize}
  \end{proposition}
  
  \begin{proof}
    (a)\\
    For a fixed topological space $D$, the assignment
    \[
       C(D,\cdot) \colon  \mathbf{TOP} \rightarrow  \mathbf{TOP}    
    \]
    is a multiplicative functor, sending topological groups to topological groups.\footnote{See \cite[Chapter 3.4]{Engel} or \cite[Appendix A]{Hatcher} for more on the compact-open topology.}
    It remains to show that for $\gamma\in C(D,G)$ we have 
    \[
       C(D,f) \bigl(  \gamma\cdot C(D,w)      \bigr) =C(D,f)\bigl(  \gamma\bigr)  \cdot \gamma\,.
    \]
    Since all group operations are pointwise, we easily check for each $d\in D$:
    \begin{align*}
      C(D,f) \bigl(  \gamma\cdot C(D,w)      \bigr) (d)
      &=  f\circ \bigl(  \gamma\cdot C(D,w)      \bigr) (d)
      \\&=  f\bigl(  \gamma(d) \cdot w      \bigr)       
      \\&=  f\bigl(  \gamma(d)\bigr)  \cdot \gamma(d)      \\&=  \left(C(D,f)\bigl(  \gamma\bigr)  \cdot \gamma\right(d)\,.      
    \end{align*}
    This completes the proof that $C(D,G)$ is a topological J-group.
    
    (b)\\
    By part (a), the witness of $C(D,G)$ is the constant map $C(D,w)$. The neutral element of $C(D,G)$ is a constant map $C(D,1_G)$. So, the witness and the neutral element both lie in the subspace of constant maps which is homeomorphic to $G$ and therefore connected.
    Hence the witness and the neutral element share the same connected component in $C(D,G)$ and by Corollary \ref{corollary_connected_component} applied to the topological J-group $C(D,G)$, the group $C(D,G)$ is connected.
    
    Part (c) can be shown analogously.
  \end{proof} 
  
  Now we are able to show that every pathwise connected topological J-group is weakly  contractible, i.e. all homotopy groups vanish:
  \begin{proof}[Proof of Theorem \ref{thm_HOMOTOPY}]
    For each $n\in \mathbb{N}_0$ we consider the proposition
    \[
       P(n):\iff\quad              
       \framebox{For every topological J-group $G$: $\bigl( \left| \pi_0(G) \right|=1 \implies  \left| \pi_n(G) \right|=1\bigr)$}\,.
    \]
    We will show $(\forall n\in \mathbb{N}_0)\,P(n)$ by induction.
    
    As $P(0)$ is trivial, we may assume $P(n)$ is true and show $P(n+1)$.
    
    Recall that $\pi_{n+1}(G)=\pi_n(\Omega(G))$, where 
    $\Omega(G):=  \left\{ \gamma\in C( \mathbb{T},G) \ :\ \gamma(0)=1_G\right\} $ 
    is the \emph{loop group} of $G$. Since every element in $C( \mathbb{T},G)$ can be written as a product of an element of $\Omega(G)$ and a constant map, we have a decomposition
    \[
       C( \mathbb{T},G) \cong \Omega(G)\rtimes G
    \]
    as topological groups. 
    In particular, we have $C( \mathbb{T},G)\approx \Omega(G)\times G$ as topological spaces.
    Hence, we can compute:
    \begin{align*}
      \left| \pi_{n+1}(G) \right|   &=  \left| \pi_n(\Omega(G)) \right|\\
      &=  \left| \pi_n(\Omega(G)) \right|\cdot  \left| \pi_n(G) \right|\\
      &=  \left| \pi_n(\Omega(G)\times G) \right|\\
      &=  \left| \pi_n(C( \mathbb{T},G)) \right|\,.
    \end{align*}
    By Proposition \ref{proposition_mapping_J_groups} we know that $C( \mathbb{T},G)$ is a pathwise connected topological J-group and hence by $P(n)$ we have $ \left| \pi_n(C( \mathbb{T},G)) \right|=1$, which completes the proof.
  \end{proof}
  
  \begin{remark}
    We have used in the proof of Theorem \ref{thm_HOMOTOPY}
    that given a topological J-group $G$, the \emph{free loop group} $C( \mathbb{T},G)$ is a topological J-group. One could expect that the \emph{loop group} (with basepoint)
    \[
       \Omega(G)=  \left\{ \gamma\in C( \mathbb{T},G) \ :\ \gamma(0)=1_G\right\}     
    \]
    is a topological J-group as well. Unfortunately, this is not true, in general, as the following example shows:
    
    Let $G:= \mathbb{Z}\times  \mathrm{SU}_{2}\!\left(   \mathbb{C}    \right)$. Then $G$ is a topological J-group by Proposition \ref{proposition_Z_times_group}. 
    We will show that $\Omega(G)$ cannot be a topological J-group. 
    First of all, we notice that 
    \[
       \Omega(G)\cong\Omega \mathbb{Z}\times\Omega( \mathrm{SU}_{2}\!\left(   \mathbb{C}    \right))\cong\Omega( \mathrm{SU}_{2}\!\left(   \mathbb{C}    \right))\,.   
    \]
    Now, we will compute the $0$th and the $2$nd homotopy group of $\Omega(G)$. To this end, we will use the fact that $ \mathrm{SU}_{2}\!\left(   \mathbb{C}    \right)$ is homeomorphic to the $3$-sphere $S^3\subseteq \mathbb{R}^4$ (see e.g. \cite[Section 2.3]{HilgertNeeb}).
    
    We have
    \[
       \pi_0(\Omega(G))\cong\pi_0(\Omega( \mathrm{SU}_{2}\!\left(   \mathbb{C}    \right)))\cong\pi_1( \mathrm{SU}_{2}\!\left(   \mathbb{C}    \right))=  \left\{ 1 \right\} \,,
    \]
    as the $3$-sphere is simply connected (see e.g. \cite[Section 4.1]{Hatcher}). This shows that $\Omega(G)$ is pathwise connected. We further calculate:
    \[
       \pi_2(\Omega( \mathrm{SU}_{2}\!\left(   \mathbb{C}    \right))) = \pi_3( \mathrm{SU}_{2}\!\left(   \mathbb{C}    \right))\cong  \mathbb{Z}\,,
    \]
    since the $3$rd homotopy group of the $3$-sphere is isomorphic to $ \mathbb{Z}$ (again, see \cite[Section 4.1]{Hatcher} for basic facts about homotopy groups). Therefore, the group $\Omega(G)$ is pathwise connected but not weakly contractible, so it cannot be a topological J-group by Theorem \ref{thm_HOMOTOPY}.
  \end{remark}
  It should be noted that for connected Lie groups also the converse of Theorem \ref{thm_HOMOTOPY} holds, i.e. a contractible Lie group is a topological J-group (Corollary \ref{corollary_contractible_lie_groups}). However, for general topological groups, the converse is false, i.e. there are contractible topological groups that are not topological J-groups. We give an example of a contractible abelian Polish (which means completely metrizable and separable) group with the desired properties:
  \begin{example}         \label{example_boolean_group}
    Let $A:=L^1([0,1], \mathbb{Z})$ be the group of all (equivalence classes of) $L^1$-functions on the interval $[0,1]$ which take only integer values (almost everywhere) and consider the quotient $G:=A/2A$. Then $G$ is a Polish group in which every element has order at most $2$ (a so called \emph{Boolean group}), so it cannot be a topological J-group by Lemma \ref{lemma_order_of_w}.
    However, the group is contractible as the following homotopy shows:
    \[
       H \colon [0,1]\times G \rightarrow G ,\ (s,f) \mapsto f\cdot\mathbbm 1_{[0,s]}\,,    
    \]
    where $\mathbbm 1_{[0,s]}$ denotes the indicator function (characteristic function) of the interval $[0,s]$.
  \end{example}
  Note that the group $A:=L^1([0,1], \mathbb{Z})$ from Example \ref{example_boolean_group} which was used to construct the contractible counter-example is a topological J-group by Theorem \ref{thm_neeb}.
  It serves as a famous counter-example in infinite-dimensional Lie theory as it is a contractible (hence pathwise connected) closed subgroup of the Banach space $L^1([0,1], \mathbb{R})$ but has a trivial Lie algebra (see \cite[Example 14.52]{ProLieBook} and \cite[Exercise E7.17]{CompBook}).
  
  The theorems we developed so far are not enough to decide whether topological groups like $( \mathbb{Z}/3 \mathbb{Z})\times \mathrm{SO}_{3}\!\left(   \mathbb{R}    \right)$ are topological J-groups, as the underlying space is not pathwise connected (so Theorem \ref{thm_HOMOTOPY} does not apply) and the number of components is odd (so Corollary \ref{corollary_connected_component} does not apply). Furthermore, the group is not abelian, so Theorem \ref{thm_COMPACT} cannot be applied either. We will now introduce a possibility to show that $( \mathbb{Z}/3 \mathbb{Z})\times \mathrm{SO}_{3}\!\left(   \mathbb{R}    \right)$ and similar groups are not topological J-groups:
  
  \begin{theorem} \ \\ \label{thm_odd_homotopy_groups}
    Let $G$ be a topological J-group and let $m\in \mathbb{N}_0$ be given. If all the homotopy groups
    \[
       \pi_0(G),\pi_1(G),\ldots,\pi_m(G)
    \]
    are finite, then the order of all these groups must be odd.
  \end{theorem}
  \begin{proof}
    Just for this proof, let us use the notation
    \[
       \Lambda(G):=C( \mathbb{T},G)
    \]
    for the \emph{free loop group} of $G$.
    Recursively, we use the notation $ \Lambda^0(G):=G$ and $ \Lambda^{k+1}(G):= \Lambda( \Lambda^k(G))$.
    We use analog notation for $\Omega^k(G)$ for $k\in \mathbb{N}_0$.
    
    Now, using the idea of the proof of Theorem \ref{thm_HOMOTOPY}, we observe that
    $ \Lambda(G)$ is \emph{homeomorphic} to $\Omega(G)\times G$.
    Similarly, we get the homeomorphism
    \begin{align*}
      \Lambda^2(G)
      &\approx  \Lambda(\Omega(G)\times G)
      \approx  \Lambda(\Omega(G)\times  \Lambda(G)
      \\&\approx \Omega^2(G)\times \Omega(G)\times \Omega (G) \times G
      = \Omega^2(G)\times \bigl(\Omega(G)\bigr)^2\times G\,.
    \end{align*}
    One easily checks that iterated application of $ \Lambda$ yields a binomial theorem-like formula as follows:
    \[
       \Lambda^m(G) \approx \prod_{k=0}^m \bigl(  \Omega^k(G)    \bigr)^{\binom{m}{k}}\,.
    \]
    One should keep in mind, that this is only a homeomorphism and does not in general preserve the group structure.
    
    By iterated application\footnote{Alternatively, one observes that $ \Lambda^m(G)$ is isomorphic to $C( \mathbb{T}^m,G)$ as a topological group and therefore we only have to apply Proposition \ref{proposition_mapping_J_groups} once.} of Proposition \ref{proposition_mapping_J_groups}, we see that $ \Lambda^m(G)$ is a topological J-group.
    Now, we determine the number of path-components:
    \[
       \left| \pi_0\left(    \Lambda^m(G)   \right) \right|       =\prod_{k=0}^m    \left| \pi_0(\Omega^k(G)) \right|^{\binom{m}{k}}
       =\prod_{k=0}^m  \left| \pi_k(G) \right|^{\binom{m}{k}}\,.
    \]
    By assumption, this number is finite and therefore has to be odd by Corollary \ref{corollary_connected_component}. Hence, every factor $ \left| \pi_k(G) \right|$ is odd.
  \end{proof}
  
  \begin{corollary}
    For every finite group $H$ and every $n\geq3$, the topological group $H\times  \mathrm{SO}_{n}\!\left(   \mathbb{R}    \right)$ is not a topological J-group.
  \end{corollary}
  \begin{proof}
    Set $G:=H\times  \mathrm{SO}_{n}\!\left(   \mathbb{R}    \right)$ and let $m=1$. 
    Then $\pi_0(G)\cong H$ is a finite group.
    Furthermore, we have $\pi_1(G)=\pi_1( \mathrm{SO}_{n}\!\left(   \mathbb{R}    \right))\cong \mathbb{Z}/2 \mathbb{Z}$ (see e.g. \cite[Section 3.D; Exercise 2]{Hatcher}) which is again finite.
    
    Since the order of $\pi_1(G)$ is even, the group $G$ cannot be a topological J-group by Theorem \ref{thm_odd_homotopy_groups}.
  \end{proof}

  We conclude this section with the statement and proof of Theorem \ref{thm_neeb} which (together with Corollary \ref{corollary_contractible_lie_groups}) 
  can be seen as a partial converse of Theorem \ref{thm_HOMOTOPY}. 
  The author owes this theorem to Karl-Hermann Neeb who kindly gave it to him in a private communication.  
  
  We briefly recall the definition of a (locally trivial) fibre bundle:
  \begin{definition}
    Let $F\neq\emptyset$ be a topological space.
    A \emph{fibre bundle} with \emph{typical fibre} $F$ is a continuous map $ \pi \colon E \rightarrow B$ from the \emph{total space} $E$ to the \emph{base space} $B$ such that the following local triviality property is satisfied:
    For each $b\in B$ there is an open neighborhood $U\subseteq B$ of $b$ and an homeomorphism $ \phi \colon \pi^{-1}(U) \rightarrow U\times F$ such that for each $x\in\pi^{-1}(U)$ we have $\pi(x)=\pi_1(\phi(x))$, where $ \pi_1 \colon U\times F \rightarrow U$ denotes the projection onto the first component.
    
    We will use the words \emph{bundle}, \emph{fibre bundle} and \emph{local trivial fibre bundle} interchangeably.
    
    If the typical fibre $F$ is discrete, a bundle is called a \emph{covering map}.
  \end{definition}

  The following fact can be found in       \cite[§11, Lemme 4]{Dixmier}:
  \begin{lemma}\label{lemma_dixmier}
    Let $ \pi \colon E \rightarrow B$ be a locally trivial fibre bundle with typical fibre $F$. Assume that $F$ is contractible and that $B$ is paracompact. Then there exists a continuous global section, i.e. a continuous map $ s \colon B \rightarrow E$ such that $\pi\circ s =  \mathrm{id}_B$.
  \end{lemma}
  
  \begin{theorem}\label{thm_neeb}
    Let $G$ be a topological group such that the underlying space is paracompact and contractible. Furthermore, let $w\in G$ be an element such that the cyclic subgroup $ \left\langle  w   \right\rangle$ is infinite and discrete in the induced topology. Then there is a continuous map $ f \colon G \rightarrow G$ such that $(G,w,f)$ is a topological J-group.
  \end{theorem}
  \begin{proof}
    Let $E:=G\times G$ denote the product of $G$ with itself.
    The projection onto the first factor $ \pi_1 \colon E \rightarrow G$ is    \    a (trivial) fibre bundle with typical fibre $G$.
    
    It is easy to check that the map
    \[
       \phi \colon E \rightarrow E ,\ (u,v) \mapsto (uw,vu)    
    \]
    is a homeomorphism of the total space $E$. We therefore get an action of the discrete group $ \mathbb{Z}$ on $E$ via
    \[
       k.(u,v)=\phi^k(u,v)\,.
    \]
    This action is compatible with the fibre bundle structure in the following sense:
    \[
       \pi_1(k.(u,v)) = \pi_1(u,v)\cdot w^k\,.
    \]
    Since the group element $w\in G$ has infinite order by assumption, it follows that the homeomorphism $\phi$ has infinite order as well.
    
    This allows us to construct a map $ \widetilde{\pi_1} \colon E/ \mathbb{Z} \rightarrow G/ \left\langle  w   \right\rangle$ such that the following diagram commutes:
    \[
       \xymatrix{
       E \ar[d]_{\pi_1}\ar[rr]^{q_\phi} &&E/ \mathbb{Z}  \ar[d]_{\widetilde{\pi_1}}   \\
       G \ar[rr]_{q_w}&& G/ \left\langle  w   \right\rangle\,,
       }
    \]
    where $ q_\phi \colon E \rightarrow E/ \mathbb{Z}$ and $ q_w \colon G \rightarrow G/ \left\langle  w   \right\rangle$ denote the canonical quotient maps.
    
    Since---by assumption---the group $ \left\langle  w   \right\rangle$ is infinite and discrete in $G$, the map $ \widetilde{\pi_1} \colon E/ \mathbb{Z} \rightarrow G/ \left\langle  w   \right\rangle$ becomes a fibre bundle with typical fibre $G$ as well.
    
    The base space $G/ \left\langle  w   \right\rangle$ of this new bundle is paracompact since a paracompact group modulo a discrete subgroup is paracompact (see \cite[Corollary 1.5]{AntonyanQuotients}). The typical fibre is $G$ which is contractible by assumption. Hence, we can apply Lemma \ref{lemma_dixmier} to $\widetilde{\pi_1}$ and obtain a global section $ \widetilde s \colon G/ \left\langle  w   \right\rangle \rightarrow E/ \mathbb{Z}$.

    As a next step, we take a closer look at the quotient map $ q_\phi \colon E \rightarrow E/ \mathbb{Z}$. By construction it is clear that is a covering map.

    The composition $ \widetilde s\circ q_w \colon G \rightarrow E/ \mathbb{Z}$ is a continuous map defined on the simply connected space $G$ with values in the space $E/ \mathbb{Z}$ which has $E$ as a covering space. Therefore, we may lift this map to a new map $ s \colon G \rightarrow E$.
    
    \[
       \xymatrix{G\ar[d]_{s} \ar[rr]^{q_w} && G/ \left\langle  w   \right\rangle\ar[d]_{\widetilde s}\ar@/^2pc/[dd]^{ \mathrm{id}_{G/ \left\langle  w   \right\rangle}} \\
       E\ar[d]^{\pi_1}\ar[rr]_{q_\phi} && E/ \mathbb{Z} \ar[d]_{\widetilde{\pi_1}}         \\
       G \ar[rr]^{q_w} && G/ \left\langle  w   \right\rangle}
    \]
    
    The map $ s \colon G \rightarrow E=G\times G ,\ x \mapsto (s_1(x),s_2(x)$ has two component maps $s_1,s_2$. Using that $G$ is connected and $ \left\langle  w   \right\rangle$ is discrete, we can use $q_w\circ s_1 = q_w$ to conclude that there is a $\ell\in \mathbb{Z}$ such that $s_1(x)=xw^\ell$ for all $x\in G$.
    
    Finally, we will show that for all $x\in G$ we have
    \[
       s_2(xw)=s_2(x)x
    \]
    which will show that $(G,w,s_2)$ is a topological J-group.
    
    To this end, let $x\in G$ be given. 
    Since $[xw]_w=[x]_w$, we may apply $\widetilde s$ on both sides and obtain:
    \begin{align*}
      [xw]_w                   &= [x]_w                \\
      \widetilde s ([xw]_w)    &=\widetilde s([x]_w)   \\
      \left[  ( s_1(xw),s_2(xw) )   \right]_\phi &
      = \left[ (  s_1(x),s_2(x) )   \right]_\phi         \\
      \left[  ( xw^{\ell+1},s_2(xw) )   \right]_\phi &
      = \left[ (  x^\ell,s_2(x) )   \right]_\phi \,.        \\
    \end{align*}
    This means that there is an integer $k\in \mathbb{Z}$ such that
    \[
       (xw^{\ell+1},s_2(xw)) = \phi^k(x^\ell,s_2(x))\,.
    \]
    Since in the first component $\phi$ just multiplies by $w$ from the right, it follows that $k=1$ and hence
    \[
       (xw^{\ell+1},s_2(xw))=(xw^{\ell+1},s_2(x)x)\,.       \qedhere
    \]
    
  \end{proof}
  
  As a direct corollary, we obtain the following classification of topological J-groups inside the class of connected Lie groups which then will have Theorem \ref{thm_NILPOTENT} as a special case:
  \begin{corollary} \label{corollary_contractible_lie_groups}
    For a connected Lie group $G$, the following are equivalent:
    \begin{itemize}
     \item [(i)] $G$ is a topological J-group.
     \item [(ii)] The underlying space of $G$ is contractible.
     \item [(iii)] There is a simply connected solvable Lie group $H$ and an integer $k\in \mathbb{N}_0$ such that $G$ decomposes as a semidirect product
      \[
         G = H\rtimes \left(\widetilde{ \mathrm{SL}_{2}\!\left(   \mathbb{R}    \right)}\right)^k\,,
      \]
      where $\widetilde{ \mathrm{SL}_{2}\!\left(   \mathbb{R}    \right)}$ is the universal cover of the special linear group $ \mathrm{SL}_{2}\!\left(   \mathbb{R}    \right)$.
    \end{itemize}
    Furthermore, every nontrivial element can occur as a witness of $G$.
  \end{corollary}
  \begin{proof}
    (i)$\implies$(ii):\\
    Let $G$ be a connected Lie group which is a topological J-group. 
    Since connected Lie groups are automatically pathwise connected, we may apply Theorem \ref{thm_HOMOTOPY} to conclude that $G$ is weakly contractible and hence (as every Lie group has the homotopy type of a CW-complex) contractible by Whitehead's Theorem (See e.g. \cite[Theorem 4.5]{Hatcher}).
    
    (ii)$\implies$(iii):\\
    A contractible Lie group has the following Levi-decomposition:
    \[
       G = H\rtimes \left(\widetilde{ \mathrm{SL}_{2}\!\left(   \mathbb{R}    \right)}\right)^k
    \]
    where $H$ is a solvable Lie group (the radical of $G$), $k\in \mathbb{N}_0$ is a nonnegative integer and $\widetilde{ \mathrm{SL}_{2}\!\left(   \mathbb{R}    \right)}$ is the universal cover of the matrix group $ \mathrm{SL}_{2}\!\left(   \mathbb{R}    \right)$ which is the only contractible simple Lie group.\footnote{For a proof see \cite[Lemma 4.3]{HofmannNeeb}.}
    
    (iii)$\implies$(ii):\\
    A simply connected solvable group is always diffeomorphic to some Euclidean space (see \cite[Theorem 14.4.1]{HilgertNeeb}) and since $\widetilde{ \mathrm{SL}_{2}\!\left(   \mathbb{R}    \right)}$ is diffeomorphic to $ \mathbb{R}^3$, it follows that $G$ is diffeomorphic (and therefore) homeomorphic to some Euclidean space. So $G$ is contractible.
    
    (ii)$\implies$(i):\\
    By the Manifold Splitting Theorem (see \cite[Theorem 14.3.11]{HilgertNeeb}) the group $G$ is diffeomorphic to a product of some Euclidean space and a maximal compact subgroup. Since $G$ is contractible, this implies that the maximal compact subgroup has to be trivial, so $G$ does not contain any nontrivial compact subgroups.
    
    Let $w\in G\setminus  \left\{ 1_G \right\} $ be an arbitrary element. 
    By Weil's Lemma (see \cite[Proposition 7.43]{CompBook}) the group $ \left\langle  w   \right\rangle$ either has a compact closure in $G$ or is topologically isomorphic to the discrete group $ \mathbb{Z}$. Since we already know that the only compact subgroup is trivial and $w\neq1_G$, we are left with the case $ \left\langle  w   \right\rangle\cong  \mathbb{Z}$. Since a Lie group is always metrizable and therefore paracompact, the claim follows from Theorem \ref{thm_neeb}.
  \end{proof}
  
  Since every nilpotent Lie group is in particular solvable, Theorem \ref{thm_NILPOTENT} stated in the introduction directly follows from Corollary \ref{corollary_contractible_lie_groups}.

  \begin{example}\label{example_pronilpotent_group_with_trivial_center}
    Theorem \ref{thm_neeb} can also be used to give an example of a connected topological J-group with a trivial center. For each $n\in \mathbb{N}$, consider the nilpotent Lie group of strictly upper triangular matrices
    \[
       T_n:= 
       \begin{pmatrix}
         1&*&\cdots&*        \\
         0&1&\cdots&*        \\
         \vdots& &\ddots& \vdots  \\
         0&0  &\cdots      &1    
       \end{pmatrix}\subseteq \mathrm{GL}_{n}\!\left(   \mathbb{R}    \right)\,.
    \]
    The projective limit of the system $ \left( T_n \right)_{n\in \mathbb{N}}$ can be identified with the group of infinite upper triangular $( \mathbb{N}\times \mathbb{N})$-matrices.
    This is a pronilpotent pro-Lie group (see e.g. \cite[Theorem 52]{ProLieBook} for precise definitions of these terms) with trivial center, so there is no obvious candidate for a witness $w$. However, the group is contractible and contains many discrete infinite cyclic subgroups. Therefore it is a topological J-group by Theorem \ref{thm_neeb}.
    
    There are many other important (but more complicated) classes of pronilpotent pro-Lie groups with trivial center, e.g. the Butcher group from numerical analysis (see e.g.  \cite{Butcher}).

  \end{example}

 \section{Compact Abelian J-Groups} 
  \label{section_compact}
  In this section we will show Theorem \ref{thm_COMPACT}, namely that compact metrizable abelian topological J-groups are profinite.
  In order to do so, we will use the concept of an exponent sequence:
  \begin{definition}    \label{definition_exponent_sequence}
    Let $G$ be any topological group and $x\in G$. A sequence $ \left( n_k \right)_{k\in \mathbb{N}}$ in $ \mathbb{Z}\setminus   \left\{ 0 \right\} $ is called an \emph{exponent sequence} for the element $x$ if
    $x^{n_k}\to 1_G$.
    If there exists one fixed sequence $ \left( n_k \right)_{k\in \mathbb{N}}$ in $ \mathbb{Z}\setminus  \left\{ 0 \right\} $ that works for every $x\in G$, we call this sequence a \emph{pointwise exponent sequence} for $G$.
  \end{definition}
  \begin{remark}    \label{remark_constant_exponent_sequence}
    It is clear that for an element $x\in G$ of finite order $n\in \mathbb{N}$, the constant sequence $ \left( n_k=n \right)_{k\in \mathbb{N}}$ is an exponent sequence for $x$. 
  \end{remark}
  
  \begin{proposition}           \label{proposition_exponent_sequences_in_compact_groups}
    Let $G$ be a compact metrizable topological group. Then every element has an exponent sequence.
  \end{proposition}
  \begin{proof}
    Let $x\in G$. If $x$ is of finite order, then by Remark \ref{remark_constant_exponent_sequence}, the element $x$ has an exponent sequence. Let us assume now that $x$ is of infinite order, i.e. the subgroup $\langle x \rangle$ generated by $x$ is infinite.
    Let $K:=\overline{\langle x \rangle}\subseteq G$ be the closed subgroup generated by $x$ (the so called \emph{monothetic subgroup} generated by $x$).
    Since closed subgroups of compact metrizable spaces are compact and metrizable, we may assume without loss of generality, that $G=K$ and that $\langle x \rangle$ is dense in $G$.
    The identity in $G$ has a countable neighborhood basis $ \left( V_k \right)_{k\in \mathbb{N}}$ where each $V_k$ is an open subset of $G$ containing $1_G$. Since $G$ is Hausdorff, $V_k\setminus  \left\{ 1_G \right\} $ is open as well.
    
    For each $k\in \mathbb{N}$, the set $V_k\setminus  \left\{ 1_G \right\} $ is nonempty, since otherwise $  \left\{ 1_G \right\} $ would be open and $G$ therefore discrete, which is impossible since a discrete compact space is always finite and $G$ contains the infinite subset $\langle x \rangle$.
    
    Since $\langle x\rangle=  \left\{ x^n \ :\ n\in \mathbb{Z}\right\} $ is dense in $G$ and each $V_k\setminus  \left\{ 1_G \right\} $ is open and nonempty, there is an $n_k\in \mathbb{Z}$ such that $x^{n_k}\in V_k\setminus  \left\{ 1_G \right\} $. Since $x^0=1_G$, we know that $n_k\neq 0$ and since $ \left( V_k \right)_{k\in \mathbb{N}}$ is a neighborhood basis of $1_G$, it follows that $ \left( x^{n_k} \right)_{k\in \mathbb{N}}$ converges to $1_G$.
  \end{proof}
  By Proposition \ref{proposition_exponent_sequences_in_compact_groups} we know now that every element in the circle group $ \mathbb{T}$ has an exponent sequence. However, it is not true that one can find one exponent sequence that simultaneously works for all elements:
  \begin{proposition}       \label{proposition_circle_has_no_pointwise_exponent_sequence}
    The circle group $ \mathbb{T}= \mathbb{R}/ \mathbb{Z}$ does not have a pointwise exponent sequence.
  \end{proposition}
  \begin{proof}
    Although this is a direct consequence of the fact that $ \mathbb{Z}$ with its Bohr topology is sequentially discrete (see \cite[Lemma 1]{Reid1967}), we will give a proof here---which is very different\footnote{This proof also shows that there is not even an exponent sequence that converges almost everywhere.} from the one in \cite{Reid1967}.
    Let us consider the continuous function
    \[
       \gamma \colon  \mathbb{R}/ \mathbb{Z} \rightarrow  \mathbb{C} ,\ [t]_ \mathbb{Z} \mapsto e^{2\pi i t}\,.    
    \]
    Assume $ \left( n_k \right)_{k\in \mathbb{N}}$ is a pointwise exponent sequence for $ \mathbb{T}$. 
    Then for each $[t]_ \mathbb{Z}\in \mathbb{R}/ \mathbb{Z}$ we have that $ \left( [n_k t]_ \mathbb{Z} \right)_{k\in \mathbb{N}}$ converges to $[0]_ \mathbb{Z}$ and by continuity of $\gamma$ we obtain that
    $\left(\gamma([t]_ \mathbb{Z})\right)^{n_k} =  \left( e^{2\pi i n_k t} \right)_{k\in \mathbb{N}}$ converges in $ \mathbb{C}$ to $1$.
    
    This shows that $ \left( \gamma^{n_k} \right)_{k\in \mathbb{N}}$ converges pointwise to the constant $1$-function. Since $ \left| \gamma^{n_k}(t) \right|=1$ for all $[t]\in \mathbb{T}$ and $k\in \mathbb{N}$, we may apply Lebesgue's Theorem of Dominated Convergence (see e.g. \cite[Theorem 106 in Chapter 2]{HandbookOfMeasureTheory}) to $ \left( \gamma^{n_k} \right)_{k\in \mathbb{N}}$ and obtain that
    \[
       \lim_{k\to\infty} \int_0^1 e^{2\pi i n_k t} dt = \int_0^1 1 dt\,.
    \]
    Since every $n_k$ is nonzero, the integrals on the left hand side are all $0$, while the right hand side integral is equal to $1$. This contradiction completes the proof.
  \end{proof}
  Now we still have to connect the idea of exponent sequences to the concept of topological J-groups:
  \begin{proposition}       \label{proposition_J_group_exponent_sequence}
    Let $(G,w,f)$ be a topological J-group with $w$ in the center of $G$.
    A sequence $ \left( n_k \right)_{k\in \mathbb{N}}$ in $ \mathbb{Z}\setminus  \left\{ 0 \right\} $ is a pointwise exponent sequence for $G$ if and only if $ \left( n_k \right)_{k\in \mathbb{N}}$ is an exponent sequence for $w$.
  \end{proposition}
  \begin{proof}
    Assume $ \left( n_k \right)_{k\in \mathbb{N}}$ is an exponent sequence for $w$. 
    If we apply Lemma \ref{lemma_x_plus_nw} to $x=1_G$ we have that
    \[
       f(w^{n_k}) = f(1_G)\cdot w^{\binom{n_k}{2}}\,.
    \]
    Since $ \left( w^{n_k} \right)_{k\in \mathbb{N}}$ converges to $1_G$ in $G$ and since $f$ is continuous, this implies that the sequence $ \left( w^{\binom{n_k}{2}} \right)_{k\in \mathbb{N}}$ converges to $1_G$ in $G$.
    
    Apply now Lemma \ref{lemma_x_plus_nw} to an arbitrary $x\in G$ to obtain
    \[
       f(x\cdot w^{n_k}) = f(x)\cdot x^{n_k}\cdot w^{\binom{n_k}{2}}\,.
    \]
    Since $ \left( w^{n_k} \right)_{k\in \mathbb{N}}$ and $ \left( w^{\binom{n_k}{2}} \right)_{k\in \mathbb{N}}$ both converge to $1_G$, we can deduce that $ \left( x^{n_k} \right)_{k\in \mathbb{N}}$ also converges to $1_G$.
    This shows that $ \left( n_k \right)_{k\in \mathbb{N}}$ is a pointwise exponent sequence for $G$. The other implication is clear.
  \end{proof}
  In order to prove Theorem \ref{thm_COMPACT} we need to reduce the case of an arbitrary compact metrizable group to the circle group:
  
  \begin{proposition}           \label{proposition_compact_exponent_sequence_implies_profinite}
    A compact topological group\footnote{In an earlier arXiv-version of this article, this proposition was only shown for abelian groups using Pontryagin-duality. This new argument is more general and works without the commutativity assumption.} with a pointwise exponent sequence is profinite. 
  \end{proposition}
  \begin{proof}
    Consider the set of all closed normal subgroups $N\subseteq G$ such that $G/N$ is a Lie group. 
    By \cite[Corollary 2.43]{CompBook}, the group $G$ is the projective limit of all those Lie group quotients.
    If all Lie group quotients are $0$-dimensional, they are finite and the group $G$ is a projective limit of finite groups, and therefore profinite.
    
    Assume by contradiction that there is at least one Lie group quotient $L:=G/N$ of dimension at least $1$. The given pointwise exponent sequence of $G$ is also a pointwise exponent sequence for $L$.
    
    The identity component of the Lie group $L$ is open in $L$ which implies that the identity component is of positive dimension as well. 
    By \cite[Theorem 6.30]{CompBook}, the identity component is the union of all maximal tori, so in particular, there is a torus $T\subseteq L$ of positive dimension contained in $G$. Since every torus of positive dimension contains a torus of dimension $1$ as a closed subgroup, this implies that there is a copy of the circle group $ \mathbb{T}= \mathbb{R}/ \mathbb{Z}$ inside $L$ and therefore $ \mathbb{T}$ has a pointwise exponent sequence which is impossible by Proposition \ref{proposition_circle_has_no_pointwise_exponent_sequence}.
  \end{proof}
  
  We finish this section with the promised proof of Theorem \ref{thm_COMPACT}:
  \begin{proof}[Proof of Theorem \ref{thm_COMPACT}]
    Let $(G,w,f)$ be a compact abelian metrizable topological J-group.
    Since $G$ is compact and metrizable, we get an exponent sequence $ \left( n_k \right)_{k\in \mathbb{N}}$ for the witness $w\in G$ by Proposition \ref{proposition_exponent_sequences_in_compact_groups}. 
    Since $G$ is abelian, the witness lies in the center and so, by Proposition \ref{proposition_J_group_exponent_sequence}, this sequence is a pointwise exponent sequence for $G$. 
    Finally, by Proposition \ref{proposition_compact_exponent_sequence_implies_profinite}, the group $G$ is profinite.
  \end{proof}

 \section{Topological J-Rings and modules}
  \label{section_j_rings}
  A natural source for abelian topological groups is to consider additive groups of topological rings.
  \begin{definition}\label{definition_j_ring}
    A \emph{topological J-ring} $(R,f)$ is a topological ring $R$ together with a continuous self-map $ f \colon R \rightarrow R$ such that
    $((R,+),1_R,f)$ is a topological J-group. 
    Here $(R,+)$ denotes the additive group of the ring $R$.
  \end{definition}
  Recall that all rings are unital.
  
  \begin{remark}    \label{remark_products_of_J_rings}
    It is easy to see that arbitrary products of topological J-rings are topological J-rings.
  \end{remark}
  
  \begin{proposition}\label{lemma_binomial_j_ring}
    Let $R$ be a topological ring. Then each of the following statements implies the next one:
    \begin{enumerate}[(i)]
     \item The characteristic of $R$ is an odd number.
     \item The element $2_R:=1_R+1_R$ is invertible in $R$.
     \item The ring $R$ is a topological J-ring.
     \item The group $(R,+)$ is a topological J-group.
     \item The characteristic of $R$ is an odd number or zero.
    \end{enumerate}
    In particular, for topological rings with positive characteristic we have the equivalence:
    \[
       R \text{ is a topological J-ring}\iff  \mathrm{char}\left(R\right)\text{ is odd}\,. 
    \]
    
  \end{proposition}
  \begin{proof}
    Let $c:= \mathrm{char}\left(R\right)\in \mathbb{N}_0$.
    If $c=2k+1$ with $k\in  \mathbb{N}_0$, then the element $(k+1)1_R$ is the multiplicative inverse of $2_R$. This shows that (i) implies (ii).
    
    If $2_R$ is invertible, then the binomial coefficient map
    \[
       f \colon R \rightarrow R ,\ x \mapsto \binom{x}{2}=(2_R)^{-1}\cdot x\cdot (x-1_R)    
    \]
    is well-defined and continuous.
    It satisfies
    \[
       f(x+1_R)  = (2_R)^{-1}(x+1_R)\cdot x 
       = (2_R)^{-1}(x-1_R)\cdot x+x
       = f(x) + x\,.
    \]
    This shows that (iii) follows from (ii).
    
    Since for the implication (iii)$\implies$(iv) there is nothing to show, we now assume (iv), i.e. that $((R,+),w,f)$ is a topological J-group and show that $ \mathrm{char}\left(R\right)$ is zero or an odd number.
    
    If $ \mathrm{char}\left(R\right)=0$, we are finished, so assume that $c= \mathrm{char}\left(R\right)>0$. This implies that $c$ is the maximal order of all elements of $(R,+)$. So by Lemma \ref{lemma_order_of_w} the order of $w$ is equal to $c$ and---again by Lemma \ref{lemma_order_of_w}---this is an odd number.
  \end{proof}
  
  \begin{remark}
    In Lemma \ref{lemma_binomial_j_ring} the implications  ``(i)$\implies$(ii)'' and  ``(ii)$\implies$(iii)'' are no equivalences as the examples $R= \mathbb{R}$ and $R= \mathbb{Z}$ show. 
    It is not known to the author if the implication from (iii) to (iv) cannot be reversed or if (iii) and (iv) are actually equivalent.
    
    The implication  ``(iv)$\implies$(v)'' cannot be reversed as the following instructive example shows:
    Consider the ring $R:= \mathbb{Z}_3\times ( \mathbb{Z}/2 \mathbb{Z})$ which has characteristic $0$. We will see that $R$ is not a topological J-ring. We assume by contradiction that $(R,f)$ was a topological J-ring and apply Lemma \ref{lemma_x_plus_nw} to $x=0_R$ and $n=2\cdot 3^k$ where $k\in \mathbb{N}$:
    \begin{align*}
      f( 2\cdot 3^k \cdot 1_R) 
      &=f( 0_R + 2\cdot 3^k \cdot 1_R) 
      \\&=f(0_R) + \underbrace{2\cdot 3^k \cdot 0_R}_{=0_R} + \underbrace{\binom{2\cdot 3^k}{2}}_{=3^k(2\cdot 3^k-1)} 1_R
      \\&=f(0_R) +3^k(2\cdot 3^k-1)\cdot\left(1,\left[1\right]_2   \right)
      \\&=f(0_R) +\left( 3^k(2\cdot 3^k-1), \left[3^k(2\cdot 3^k-1)\right]_2   \right)
      \\&=f(0_R) +\left( 3^k(2\cdot 3^k-1), [1]_2   \right)\,.
    \end{align*}
    Now, since $ \left( 3^k \right)_{k\in \mathbb{N}}$ converges to $0$ in $ \mathbb{Z}_3$, we can use continuity of the map $f$ and deduce that
    $f(2\cdot 3^k \cdot 1_R)=f(2\cdot 3^k,[0]_2)$ converges to $f(0_R)$. However, $f(0_R) +\left( 3^k(2\cdot 3^k-1), [1]_2   \right)$ converges to $f(0_R)+(0,[1]_2)$ which is a different value. So, we arrive at a contradiction.
    
    Note that the similar looking ring $ \mathbb{Z}_2\times( \mathbb{Z}/2 \mathbb{Z})$ is a topological J-ring although it has elements of order $2$.
    
    There are also topological rings of characteristic $0$, where the author does not know whether they are topological J-rings or not. My favourite example (see Question \ref{question_INFINITE_PRODUCT}) is
    \[
       R = \prod_{j=1}^\infty  \mathbb{Z}/(2^k \mathbb{Z})\,.
    \]
    Interestingly, it contains the $2$-adic integers $ \mathbb{Z}_2$ as a closed subring which is a topological J-ring by Example \ref{example_J_ring}.
  \end{remark}
  For the case of fields, the situation is much easier than for rings:
  \begin{lemma} \label{lemma_j_field}
    A topological field is a topological J-ring if and only if the characteristic is not equal to $2$.
  \end{lemma}
  \begin{proof}
    In a field $ \mathbb{K}$ of characteristic $0$, the element $2_ \mathbb{K}=1_ \mathbb{K}+1_ \mathbb{K}\neq0_ \mathbb{K}$ is invertible and therefore, the statement follows from Lemma \ref{lemma_binomial_j_ring}
  \end{proof}
  The following is an easy observation:
  \begin{lemma} \label{lemma_j_subring}
    Let $(R,f)$ be a topological J-ring and let $S$ be a subring of $R$ such that $f(S)\subseteq S$. Then $S$ is a topological J-ring, as well.\qed
  \end{lemma}

  \begin{example}\label{example_J_ring}
    We will list a few examples of topological J-rings:
    \begin{itemize}
     \item The fields of real numbers $ \mathbb{R}$, complex numbers $ \mathbb{C}$ and $p$-adic rationals $ \mathbb{Q}_p$ (for every prime number $p$) have characteristic zero and hence are topological J-rings by Lemma \ref{lemma_j_field}.
     \item The ring of integers $ \mathbb{Z}$ with the discrete topology is a J-ring by Lemma \ref{lemma_j_subring}---or by an easy direct calculation---or by Proposition \ref{proposition_infinite_order_J_group}.
     \item For $m\in \mathbb{N}$ we have that $ \mathbb{Z}/m \mathbb{Z}$ is a topological J-ring if and only if $m$ is odd (Lemma \ref{lemma_binomial_j_ring}).
     \item For every prime number $p$ the topological ring of $p$-adic integers $ \mathbb{Z}_p$ is a  topological J-ring. 
      For $p\neq2$, this follows directly from Lemma \ref{lemma_binomial_j_ring}, since $2$ is invertible in $ \mathbb{Z}_p$. For $p=2$, the number $2$ is not invertible but we can use the fact that $ \mathbb{Q}_2$ is a topological J-ring and then apply Lemma \ref{lemma_j_subring}.
    \end{itemize}
  \end{example}
  
  We are now ready to prove Theorem \ref{thm_PROFINITE} stated in the introduction:
  \begin{proof}[Proof of Theorem \ref{thm_PROFINITE}]
    We begin with part (a):\\
    Let $G$ be a torsion-free abelian profinite group. Then by \cite[Theorem 4.3.3]{RibesZalesskii}, the group $G$ is topologically isomorphic to
    \[
       G \cong \prod_p \left(    \mathbb{Z}_p   \right)^{\mathfrak m(p)}\,,
    \]
    where the product ranges over all prime numbers and each $\mathfrak m(p)$ is some cardinal number.
    Since $ \mathbb{Z}_p$ is a topological J-ring for all prime numbers (including $p=2$) by Example \ref{example_J_ring}, the assertion follows from Remark \ref{remark_products_of_J_rings}.
    
    Now, we come to part (b):\\
    Let $G$ be a torsion abelian profinite group. Then by \cite[Corollary 4.3.9]{RibesZalesskii}, there is a finite set $\Pi$ of prime numbers and a natural number $e\in \mathbb{N}$ such that we have the isomorphism
    \[
       G \cong \prod_{p\in\Pi} \left( \prod_{i=1}^e\left(   \mathbb{Z}/p^i \mathbb{Z}   \right)^{\mathfrak m(p,i),}\right)\,,
    \]
    where each $\mathfrak m(p,i)$ is a cardinal which may be zero.

    Furthermore, \cite[Corollary 4.3.9]{RibesZalesskii} tells us that there is a finite exponent $N\in \mathbb{N}$ such that $x^N=1_G$ for all $x\in G$. We may choose $N\in \mathbb{N}$ to be minimal with this property.
    
    The exponent $N$ may be calculated as
    \[
       N := \prod_{p\in\Pi}p^{\max  \left\{ i\in    \left\{ 1,\ldots,e \right\}  \ :\ \mathfrak m(p,i)\neq0\right\} }\,, 
    \]
    where we set $\max\emptyset:=0$.
    
    Since every cardinal $\mathfrak m(p,i)$ is allowed to be $0$, we may assume that $2\in\Pi$ and rewrite the isomorphism as
    \[
       G \cong \underbrace{\prod_{i=1}^e\left(   \mathbb{Z}/2^i \mathbb{Z}   \right)^{\mathfrak m(2,i),}}_{R_2}\times \underbrace{ \prod_{p\in\Pi\setminus  \left\{ 2 \right\} } \left( \prod_{i=1}^e\left(   \mathbb{Z}/p^i \mathbb{Z}   \right)^{\mathfrak m(p,i),}\right)}_{R_{>2}}=R_2\times R_{>2}\,.
    \]
    The topological ring $R_{>2}$ is a topological J-ring by Lemma \ref{lemma_binomial_j_ring} and Remark \ref{remark_products_of_J_rings}. So everything depends on the first factor $R_2$.
    
    If the cardinals $\mathfrak m(2,1),\ldots,\mathfrak m(2,e)$ are all zero, then the exponent $N$ is an odd number and $R_2=1$ which means that $G$ is isomorphic to the additive group of topological J-ring and therefore $G$ is a topological J-group. Furthermore, in this case, the order of every element divides $N$ and therefore all elements of $G$ have odd order.
    
    If there is at least one cardinal $\mathfrak m(2,1),\ldots,\mathfrak m(2,e)$ which is nonzero, then there is an element of even order, and therefore, $N$ has to be an even number. By Lemma \ref{lemma_order_of_w}, an abelian group cannot be a topological J-group if the maximal order is an even number.
  \end{proof}
  Now that we dealt with topological J-rings (and fields), let us continue with modules (and vector spaces) over those:
  \begin{remark}
    Let $R$ be a topological ring. 
    Recall that a topological $R$-module is a module over the ring $R$, together with a topology such that it becomes an abelian topological group $M$ and that the scalar multiplication
    \[
       R\times M \rightarrow M : (r,v) \mapsto r\cdot v    
    \]
    is continuous.
    
    It would be nice if one could say that every topological $R$-module over a topological J-ring $R$ is a topological J-group.
    Unfortunately this is not true, as every abelian topological group is a $ \mathbb{Z}$-module, and $ \mathbb{Z}$ is a topological J-ring, but the group $ \mathbb{Z}/2 \mathbb{Z}$ is not a J-group, for example. However, the following is true:
  \end{remark}
  
  \begin{proposition} \label{proposition_module_projection}
    Let $(R,f_R)$ be a topological J-ring and let $M$ be a topological $R$-module. If there is an element $w\in M$ and a continuous function $  \phi \colon M \rightarrow R$
    such that 
    \[
       \phi(x+w)= \phi(x)+1_R\text{ holds for all x}\in M\,,
    \]
    then there is an $ f_M \colon M \rightarrow M$ such that $(M,w,f_M)$ is a topological J-group.
  \end{proposition}
  \begin{proof}
    We set
    \[
       f_M \colon M \rightarrow M ,\ x \mapsto  \phi(x)\bigl( x- \phi(x)w\bigr)+f_R( \phi(x))w    
    \]
    and check that the defining property of a topological J-group is satisfied:
    \begin{align*}
      f_M(x+w)  &=     \phi(x+w)\bigl( x+w- \phi(x+w)w\bigr)
      +f_R( \phi(x+w))w
      \\&=    ( \phi(x)+1_R)\bigl( x+w-( \phi(x)+1_R)w\bigr)+f_R( \phi(x)+1_R)w
      \\&=    ( \phi(x)+1_R)\bigl( x- \phi(x)w\bigr)
      +(f_R( \phi(x))+ \phi(x))w
      \\&=     \phi(x)\bigl( x- \phi(x)w\bigr)+x- \phi(x)w
      +f_R( \phi(x))w+ \phi(x)w
      \\&=     \phi(x)\bigl( x- \phi(x)w\bigr)+f_R( \phi(x))w + x
      \\&=    f_M(x) + x\,. \qedhere
    \end{align*}
  \end{proof}
  
  Proposition \ref{proposition_module_projection} has an analog for nilpotent Lie algebras:
  In a nilpotent Lie algebra $ \mathfrak{g}$ the \emph{Baker-Campbell-Hausdorff} series (\emph{BCH}-series) which exists as a formal power series on each Lie algebra (provided that the characteristic of the base field is zero) is just a polynomial map that is defined everywhere:
  \[
     * \colon  \mathfrak{g}\times \mathfrak{g} \rightarrow  \mathfrak{g} ,\ (x,y) \mapsto x*y:=x+y+\frac12[x,y]+\cdots  
  \]
  See \cite[Definition 1, Ch. II, §6]{MR1728312} or \cite[Definition IV.1.3]{MR2261066} for a precise definition of this series.
  
  It is well-known that a nilpotent Lie algebra, together with this BCH-multiplication becomes a group with neutral element $0$ and inversion $- \mathrm{id}_ \mathfrak{g}$.
  
  By definition of the BCH-series, it is clear that $x*y=x+y$ if $[x,y]=0$.
  
  If the Lie algebra $ \mathfrak{g}$ is further a topological Lie algebra, this group $( \mathfrak{g},*)$ is a topological group as the map $ * \colon  \mathfrak{g}\times \mathfrak{g} \rightarrow  \mathfrak{g}$ is clearly continuous.

  \begin{proposition}\ \\\label{proposition_nilpotent_J_group_lie_algebra_version}
    Let $ \mathfrak{g}$ be a nilpotent topological Lie algebra over a topological field $ \mathbb{K}$ of characteristic $0$. Let $w\in \mathfrak{g}$ be an element in the center of $ \mathfrak{g}$. Assume furthermore that there is a continuous map $  \phi \colon  \mathfrak{g} \rightarrow  \mathbb{K}$ such that
    \[
       (\forall x\in \mathfrak{g})\  \phi(x+w)= \phi(x)+1_ \mathbb{K}\,.
    \]
    Then the topological group $( \mathfrak{g},*)$ is a topological J-group with witness $w$.
  \end{proposition}
  \begin{proof}
    Let us define
    \[
       f \colon  \mathfrak{g} \rightarrow  \mathfrak{g} ,\ x \mapsto  \phi(x)(x- \phi(x)w)+\binom{ \phi(x)}{2}w\,.
    \]
    Since $w$ is in the center of $ \mathfrak{g}$, we have $[w,x]=0$ for all $x\in \mathfrak{g}$. 
    It follows by bilinearity of the Lie bracket that
    \[
       [f(x),x]=0 \quad \text{ for all }x\in  \mathfrak{g}\,.
    \]
    In particular, we have 
    \[
       x*w=x+w\quad \and \quad f(x)*x=f(x)+x\quad\text{ for all } x\in \mathfrak{g}\,.
    \]
    With this in mind, it suffices to check that
    \[
       f(x+w)=f(x)+w\,,
    \]
    which is an easy calculation, identical to that in the proof of Proposition \ref{proposition_module_projection}.
  \end{proof}
  
  \begin{remark}
    Since every simply connected nilpotent Lie group is of the form $( \mathfrak{g},*)$ of Proposition \ref{proposition_nilpotent_J_group_lie_algebra_version} and has a nontrivial center, this gives an alternative proof of Theorem \ref{thm_NILPOTENT}.
  \end{remark}

  \begin{lemma}\label{lemma_ring_times_module}
    Let $(R,f)$ be a topological J-ring and let $M$ be a topological $R$-module.    
    Then the product $R\times M$ is a topological J-group with witness $(1_R,0_M)$ and self-map
    \[
       \widetilde f \colon R\times M \rightarrow R\times M ,\ r,v \mapsto (f(r),r\cdot v)\,.    
    \]
    
  \end{lemma}
  \begin{proof}
    The product $R\times M$ is a topological $R$-module and the projection onto the first component satisfies the assumptions of Proposition \ref{proposition_module_projection}. Alternatively, one can check directly that the given map satisfies the J-group property.
  \end{proof}
  
  If we apply Lemma \ref{lemma_ring_times_module} to the ring $ \mathbb{Z}$, we obtain that for every abelian topological group $G$, the product
  \[
     \mathbb{Z}\times G
  \]
  is a topological J-group. It turns out that for this to hold, we do not need $G$ to be abelian:
  \begin{proposition}\label{proposition_Z_times_group}
    \begin{itemize}
     \item [(a)]     Let $G$ be a topological group and $w\in G$ be an element in the center of $G$.
      If there exists a continuous function $  \phi \colon G \rightarrow  \mathbb{Z}$ with $ \phi(x\cdot w)= \phi(x)+1$ for each $x\in G$, then $G$ is a topological J-group with witness $w$.
     \item [(b)] For every topological group $H$, the product $ \mathbb{Z}\times H$ is a topological J-group
      with witness $(1,1_H)$ and self-map
      \[
         f \colon  \mathbb{Z}\times H \rightarrow  \mathbb{Z}\times H ,\ (k,u) \mapsto \left(\binom{k}{2}, u^k\right)\,.
      \]
      
    \end{itemize}
  \end{proposition}
  \begin{proof}
    (a)\\
    The proof is more or less identical to the proof of Proposition \ref{proposition_module_projection}:
    We set
    \[
       f \colon G \rightarrow G ,\ x \mapsto \bigl( x\cdot w^{- \phi(x)}\bigr)^{ \phi(x)}\cdot w^{\binom{ \phi(x)}{2}}    
    \]
    and check that the defining property of a topological J-group is satisfied:
    \begin{align*}
      f(x\cdot w)  &=    \bigl( (x\cdot w)\cdot w^{- \phi(x\cdot w)}\bigr)^{ \phi(x\cdot w)}
      \cdot w^{\binom{ \phi(x+w)}{2}}
      \\  &=    \bigl( x\cdot w\cdot w^{- \phi(x)-1}\bigr)^{ \phi(x)+1}
      \cdot w^{\binom{ \phi(x)+1}{2}}
      \\  &=    \bigl( x\cdot  w^{- \phi(x)}\bigr)^{ \phi(x)+1}
      \cdot w^{\binom{ \phi(x)}{2}+ \phi(x)}
      \\  &=    \bigl( x\cdot  w^{- \phi(x)}\bigr)^{ \phi(x)}\cdot \bigl( x\cdot  w^{- \phi(x)}\bigr)
      \cdot w^{\binom{ \phi(x)}{2}}\cdot w^{ \phi(x)}
      \\  &=    \bigl( x\cdot w^{- \phi(x)}\bigr)^{ \phi(x)}\cdot x \cdot w^{\binom{ \phi(x)}{2}}
      \\  &=    \bigl( x\cdot w^{- \phi(x)}\bigr)^{ \phi(x)}\cdot w^{\binom{ \phi(x)}{2}}\cdot x
      \\  &= f(x)\cdot x\,.
    \end{align*}
    Note that in the very end, $x$ and $w^{\binom{ \phi(2)}{2}}$ commute since $w$ is assumed to be in the center of the group---a fact that is only used here.
    
    (b)\\
    Let $w:=(1,1_G)$, where $1$ denotes the number one in $ \mathbb{Z}$ and $1_G$ is the neutral element in $G$.
    Then the projection onto the first factor satisfies the hypothesis of part (a). Alternatively, one can check that the given formula for $f$ defines a J-group structure.
  \end{proof}
  
  \begin{remark}[Subgroups of $ \mathbb{R}$] \label{remark_subgroups_of_R}
    As an application of Proposition \ref{proposition_Z_times_group}(a) we will show that every subgroup $G$ of $( \mathbb{R},+)$ is a topological J-group with the induced topology and that every nonzero group element can be a witness. 
    Since we already know that $G= \mathbb{R}$ and $G=  \left\{ 0 \right\} $ are topological J-groups, we may assume that there is an element $w\in G\setminus  \left\{ 0 \right\} $ and there is a real number $\alpha\in \mathbb{R} \setminus G$. Now, the map
    \[
       \phi \colon G \rightarrow  \mathbb{Z} ,\ x \mapsto  \left\lfloor\frac{x-\alpha}{w} \right\rfloor\,,    
    \]
    where $\lfloor\cdot\rfloor$ is the \emph{floor function},
    satisfies the hypotheses of Proposition \ref{proposition_Z_times_group}(a). 
  \end{remark}
  \begin{question}
    Is it true that every subgroup of $ \mathbb{R}^2$ with the induced topology is a topological J-group?
    One would like to use a similar idea as in Remark \ref{remark_subgroups_of_R}, but the fact that there are proper and dense subgroups of $ \mathbb{R}^2$ which are connected (see e.g. \cite{Maehara1986}) makes it difficult to adjust the argument.
  \end{question}

  We will end this section with a proof of Theorem \ref{thm_TVS} stated in the introduction: Every topological vector space which satisfies one of the following properties is  a topological J-group:
  \begin{itemize}
   \item $E$ is locally convex.
   \item $E$ is paracompact.
   \item The topological dual of $E$ is not trivial.
  \end{itemize}
  \begin{proof}[Proof of Theorem \ref{thm_TVS}]
    If $E$ is paracompact, then the claim follows directly from Theorem \ref{thm_neeb}. If $E=  \left\{ 0 \right\} $, there is nothing to prove. If $E\neq  \left\{ 0 \right\} $ is locally convex, then by the Hahn-Banach-Theorem, the topological dual separates the points, so the topological dual is nontrivial.
    Finally, if the topological dual is nontrivial, there is a continuous linear functional $\phi\neq0$. This means that there is an element $w\in E$ such that $\phi(w)=1$. 
    The claim then follows from Proposition \ref{proposition_module_projection} applied to $\phi$.
  \end{proof}

 \section*{Acknowledgements}
  I want to thank Georg Frenck for a very fruitful discussion about homotopy groups during the conference \emph{Curvature and Global Shape 2021} in Münster/Germany.
  Furthermore I want to thank Dominik Bernhardt and Johannes Flake for helpful remarks on an earlier version of this article.
  Special thanks go to Karl-Hermann Neeb for Theorem \ref{thm_neeb}. Last but not least I want to thank my teacher Karl-Heinrich Hofmann who taught me the beauty of the theory of topological groups.
  \phantomsection
  \addcontentsline{toc}{section}{References}
  \bibliographystyle{new}
  \bibliography{Literatur}

\end{document}